\documentclass[10pt]{amsart}
\usepackage{amssymb,MnSymbol}
\usepackage{amsthm,amsmath}
\usepackage{cite}

\title{Strongly barycentrically associative and preassociative functions}

\author{Jean-Luc Marichal}
\address{Mathematics Research Unit, FSTC, University of Luxembourg \\
6, rue Coudenhove-Kalergi, L-1359 Luxembourg, Luxembourg} \email{jean-luc.marichal[at]uni.lu}

\author{Bruno Teheux}
\address{Mathematics Research Unit, FSTC, University of Luxembourg \\
6, rue Coudenhove-Kalergi, L-1359 Luxembourg, Luxembourg} \email{bruno.teheux[at]uni.lu}

\date{December 22, 2015}

\begin{document}

\theoremstyle{plain}
\newtheorem{theorem}{Theorem}[section]
\newtheorem{lemma}[theorem]{Lemma}
\newtheorem{proposition}[theorem]{Proposition}
\newtheorem{corollary}[theorem]{Corollary}
\newtheorem{fact}[theorem]{Fact}
\newtheorem*{main}{Main Theorem}

\theoremstyle{definition}
\newtheorem{definition}[theorem]{Definition}
\newtheorem{example}[theorem]{Example}

\theoremstyle{remark}
\newtheorem*{conjecture}{Conjecture}
\newtheorem{remark}{Remark}
\newtheorem*{claim}{Claim}

\newcommand{\N}{\mathbb{N}}
\newcommand{\Q}{\mathbb{Q}}
\newcommand{\R}{\mathbb{R}}

\newcommand{\ran}{\mathrm{ran}}
\newcommand{\dom}{\mathrm{dom}}
\newcommand{\id}{\mathrm{id}}
\newcommand{\med}{\mathrm{med}}
\newcommand{\ofo}{\mathrm{ofo}}
\newcommand{\Ast}{\boldsymbol{\ast}}

\newcommand{\bfu}{\mathbf{u}}
\newcommand{\bfv}{\mathbf{v}}
\newcommand{\bfw}{\mathbf{w}}
\newcommand{\bfx}{\mathbf{x}}
\newcommand{\bfy}{\mathbf{y}}
\newcommand{\bfz}{\mathbf{z}}

\newcommand{\length}[1]{{\vert #1 \vert}}

\newcommand\restr[2]{{
  \left.\kern-\nulldelimiterspace 
  #1 
  \right|_{#2} 
  }}

\begin{abstract}
We study the property of strong barycentric associativity, a stronger version of barycentric associativity for functions with indefinite arities. We introduce and discuss the more general property of strong barycentric preassociativity, a
generalization of strong barycentric associativity which does not involve any composition of functions. We also provide a generalization of Kolmogoroff-Nagumo's characterization of the quasi-arithmetic mean functions to strongly barycentrically preassociative functions.
\end{abstract}



\maketitle

\section{Introduction}

Let $X$ and $Y$ be arbitrary nonempty sets. Throughout this paper we regard tuples $\bfx$ in $X^n$ as $n$-strings over $X$. We let $X^*=\bigcup_{n \geqslant 0} X^n$ be the set of all strings over $X$, with the convention that $X^0=\{\varepsilon\}$ (i.e., $\varepsilon$ denotes the unique $0$-string on $X$). We denote the elements of $X^*$ by bold roman letters $\bfx$, $\bfy$, $\bfz$, etc. If we want to stress that such an element is a letter of $X$, we use non-bold italic letters $x$, $y$, $z$, etc. The \emph{length} of a string $\bfx$ is denoted by $|\bfx|$. For instance, $|\varepsilon|=0$. We endow the set $X^*$ with the concatenation operation, for which $\varepsilon$ is the neutral element. For instance, if $\bfx\in X^m$ and $y\in X$, then $\varepsilon\bfx y=\bfx y\in X^{m+1}$. Moreover, for every string $\bfx$ and every integer $n\geqslant 0$, the power $\bfx^n$ stands for the string obtained by concatenating $n$ copies of $\bfx$. In particular we have $\bfx^0=\varepsilon$.

As usual, a map $F\colon X^n\to Y$ is said to be an \emph{$n$-ary function} (an \emph{$n$-ary operation on $X$} if $Y=X$). Also, a map $F\colon X^*\to Y$ is said to be a \emph{variadic function} (a \emph{variadic operation on $X$} if $Y=X\cup\{\varepsilon\}$, a \emph{string function on $X$} if $Y=X^*$; see \cite{LehMarTeh}). For every variadic function $F\colon X^*\to Y$ and every integer $n\geqslant 0$, we denote by $F_n$ the \emph{$n$-ary part} $F|_{X^n}$ of $F$. Finally, a variadic function $F\colon X^n\to Y$ is said to be \emph{$\varepsilon$-standard} \cite{MarTeh2} if $\varepsilon\in Y$ and
$$
F(\bfx)=\varepsilon \quad\Leftrightarrow\quad \bfx=\varepsilon.
$$

Recall that a variadic operation $F\colon X^*\to X\cup\{\varepsilon\}$ is said to be \emph{barycentrically associative} (or \emph{B-associative} for short) \cite{MarTeh3} if it satisfies the equation
$$
F(\bfx\bfy\bfz) ~=~ F(\bfx F(\bfy)^{|\bfy|}\bfz),\qquad \bfx\bfy\bfz\in X^*.
$$
B-associativity (also known as \emph{decomposability} \cite{FodRou94,GraMarMesPap09}) was essentially introduced in 1909 by Schimmack~\cite{Sch09} and then used later by Kolmogoroff \cite{Kol30} and Nagumo \cite{Nag30} in a characterization of the class of quasi-arithmetic mean functions. For general background and historical notes on B-associativity, see \cite{MarTeh3}.

The following stronger version of B-associativity (also known as \emph{strong decomposability}) was introduced in \cite{Mar98,MarMatTou99}. For every $\bfx\in X^*$ and every $K\subseteq\{1,\ldots,|\bfx|\}$ (with $K=\varnothing$ if $\bfx=\varepsilon$), we denote by $\bfx|_K$ the string obtained from $\bfx$ by removing all the letters $x_i$ for which $i\in K^c=\{1,\ldots,|\bfx|\}\setminus K$. A variadic operation $F\colon X^*\to X\cup\{\varepsilon\}$ is said to be \emph{strongly barycentrically associative} (or \emph{strongly B-associative} for short) if for every $\bfx\in X^*$ and every $K\subseteq\{1,\ldots,|\bfx|\}$, we have $F(\bfx)=F(\bfx')$, where $\bfx'\in X^{|\bfx|}$ is defined by $\bfx'|_K=F(\bfx|_K)^{|\bfx|_K|}$ and $\bfx'|_{K^c}=\bfx|_{K^c}$.

For instance, if the operation $F\colon X^*\to X\cup\{\varepsilon\}$ is strongly B-associative, then it satisfies the condition
\begin{equation}\label{eq:sadf576}
F(\bfx\bfy\bfz) ~=~ F(F(\bfx\bfz)^{|\bfx|}\bfy F(\bfx\bfz)^{|\bfz|}),\qquad \bfx\bfy\bfz\in X^*.
\end{equation}

It is not difficult to see that any strongly B-associative operation $F\colon X^*\to X\cup\{\varepsilon\}$ is B-associative. The converse holds if $F$ is symmetric (i.e., $F_n$ is symmetric for every $n\geqslant 1$). However, it does not hold in general. For instance, the $\varepsilon$-standard operation $F\colon\R^*\to\R\cup\{\varepsilon\}$ defined as $F_n(\bfx)=\frac{1}{n}\sum_{i=1}^nx_i$ for every integer $n\geqslant 1$, is strongly B-associative and hence B-associative. However, the $\varepsilon$-standard operation $F\colon \R^*\to \R\cup\{\varepsilon\}$ defined by
$$
F_n(\bfx) ~=~ \sum_{i=1}^n \frac{2^{i-1}}{2^n-1}{\,}x_i{\,},\qquad n\geqslant 1,
$$
is B-associative but not strongly B-associative (see \cite[p.~37]{GraMarMesPap09}). It is also noteworthy that the strongly B-associative operations need not be symmetric. For instance the $\varepsilon$-standard operation $F\colon X^*\to X\cup\{\varepsilon\}$ defined by $F_n(\bfx)=x_1$ for every $n\geqslant 1$ is strongly B-associative, and similarly if $F_n(\bfx)=x_n$ for every $n\geqslant 1$.

Recall that a variadic function $F\colon X^*\to Y$ is said to be \emph{barycentrically preassociative} (or \emph{B-preassociative} for short) \cite{MarTeh3} if, for every $\bfx\bfy\bfy'\bfz\in X^*$, we have
$$
|\bfy| = |\bfy'| \quad\mbox{and}\quad F(\bfy) = F(\bfy') \quad\Rightarrow\quad F(\bfx\bfy\bfz) = F(\bfx\bfy'\bfz).
$$

It is easy to see that any B-associative operation $F\colon X^*\to X\cup\{\varepsilon\}$ is necessarily B-preassociative \cite{MarTeh3}. This observation motivates the introduction of the following property, which generalizes strong B-associativity.

\begin{definition}
We say that a variadic function $F\colon X^*\to Y$ is \emph{strongly barycentrically preassociative} (or \emph{strongly B-preassociative} for short) if for every $\bfx\in X^*$, every $\bfx'\in X^{|\bfx|}$, and every $K\subseteq\{1,\ldots,|\bfx|\}$, we have
$$
F(\bfx|_K) = F(\bfx'|_K) \quad\text{and}\quad \bfx'|_{K^c}=\bfx|_{K^c} \quad\Rightarrow\quad F(\bfx)=F(\bfx').
$$
\end{definition}

Just as strong B-associativity is a stronger version of B-associativity, strong B-preasso{\-}ciativity is a stronger version of B-preassociativity. However, these latter two properties are equivalent under the symmetry assumption. Also, since none of these properties involve any composition of functions, they allow us to consider a codomain Y that may differ
from the set $X\cup\{\varepsilon\}$. For instance, the length function $F\colon X^*\to\R$, defined as $F(\bfx) = |\bfx|$,
is strongly B-preassociative.

In Section 2 of this paper we investigate both strong B-associativity and strong B-preassociati{\-}vity. In particular, we provide equivalent formulations of these properties. For instance, we establish the surprising result that strong B-associativity is completely characterized by Eq.~\eqref{eq:sadf576}. We also provide factorization results for strongly B-preassociative functions. Finally, in Section 3 we recall a variant of Kolmogoroff-Nagumo's characterization of the class of quasi-arithmetic means based on the strong B-associativity property and we generalize this characterization to strongly B-preassociative functions.

The terminology used throughout this paper is the following. The domain and range of any function $f$ are denoted by $\dom(f)$ and $\ran(f)$, respectively. The identity operation on any nonempty set $E$ is denoted by $\id_E$. For every integer $n\geqslant 1$, the diagonal section $\delta_F\colon X\to Y$ of a function $F\colon X^n\to Y$ is defined as $\delta_F(x)=F(x^n)$.

\begin{remark}\label{rem:a8d5}
As already observed in \cite{MarTeh3}, if a B-associative operation $F\colon X^*\to X\cup\{\varepsilon\}$ is such that $\ran(F_n)\subseteq X$ for every $n\geqslant 1$, then the value of $F(\varepsilon)$ is unimportant in the sense that if we modify this value, then the resulting operation is still B-associative. Clearly, this observation also holds for strongly B-associative operations, B-preassociative functions, and strongly B-preassociative functions.
\end{remark}

\section{Strong barycentric associativity and preassociativity}

In this section we investigate both strong B-associativity and strong B-preassocia{\-}tivity properties. We start our investigation by showing that, surprisingly, strong B-associativity can be characterized simply by condition~\eqref{eq:sadf576}, thus providing a very concise definition of this (equational) property by means of a single equation.

\begin{proposition}\label{prop:3wordH}
A variadic operation $F\colon X^*\to X\cup\{\varepsilon\}$ is strongly B-associative if and only if it satisfies Eq.~\eqref{eq:sadf576}. Moreover, we may assume that $|\bfy|\leqslant 1$ in \eqref{eq:sadf576}.
\end{proposition}

\begin{proof}
The condition is clearly necessary. Let us show that it is also sufficient. Assuming that $F$ satisfies \eqref{eq:sadf576}, we have to prove that for every $\bfx\in X^*$ and every $K\subseteq\{1,\ldots,|\bfx|\}$, we have $F(\bfx)=F(\bfx')$, where $\bfx'\in X^{|\bfx|}$ is defined by $\bfx'|_K=F(\bfx|_K)^{|\bfx|_K|}$ and $\bfx'|_{K^c}=\bfx|_{K^c}$. Let us proceed by induction on $n=|\bfx|$. The result clearly holds for $n=0$. It also holds for $n=1$ since we have $F(x)=F(F(x))$ for any $x\in X$ (take $\bfx =x$ and $\bfy\bfz =\varepsilon$ in \eqref{eq:sadf576}). It also holds for $n=2$ since a similar argument gives $F(xy)=F(F(x)y)=F(xF(y))=F(F(xy))$ for any $x,y\in X$. Now, suppose that the result holds for any $n\geqslant 2$ and let us show that it holds for $n+1$. Let $\bfx\in X^{n+1}$, let $K\subseteq\{1,\ldots,n+1\}$, and let $\bfx'\in X^{n+1}$ be defined by $\bfx'|_K=F(\bfx|_K)^{|\bfx|_K|}$ and $\bfx'|_{K^c}=\bfx|_{K^c}$, where $K^c=\{1,\ldots,n+1\}\setminus K$. The result is trivial if $|K|=n+1$ since we have $F(\bfx)=F(F(\bfx)^{|\bfx|})$ (take $\bfy\bfz =\varepsilon$ in \eqref{eq:sadf576}). So assume that $|K|\leqslant n$ and take $k\in K^c$. Then there exist $\bfu\bfv,\bfu'\bfv'\in X^n$, with $|\bfu|=|\bfu'|$ and $|\bfv|=|\bfv'|$, such that $\bfx=\bfu x_k\bfv$ and $\bfx'=\bfu' x_k\bfv'$. We then have
$$
F(\bfx) ~=~ F(F(\bfu\bfv)^{|\bfu|}x_k F(\bfu\bfv)^{|\bfv|}) ~=~ F(F(\bfu'\bfv')^{|\bfu'|}x_k F(\bfu'\bfv')^{|\bfv'|}) ~=~ F(\bfx'),
$$
where the first and last equalities hold by \eqref{eq:sadf576} and the second equality by the induction hypothesis. This completes the proof of the proposition.
\end{proof}

The following proposition provides equivalent formulations of strong B-preassocia{\-}tivity.

\begin{proposition}\label{prop:asf8ssa}
Let $F\colon X^*\to Y$ be a variadic function. The following assertions are equivalent.
\begin{enumerate}
\item[(i)] $F$ is strongly B-preassociative.

\item[(ii)] For every $\bfx\bfx'\in X^*$ such that $|\bfx|=|\bfx'|$ and every $K\subseteq\{1,\ldots,|\bfx|\}$ we have
$$
F(\bfx|_K) = F(\bfx'|_K)\quad\mbox{and}\quad F(\bfx|_{K^c}) = F(\bfx'|_{K^c})\quad\Rightarrow\quad F(\bfx)=F(\bfx').
$$

\item[(iii)] For every $\bfx\bfx'\bfy\bfz\bfz'\in X^*$ we have
$$
|\bfx|=|\bfx'|,\quad|\bfz|=|\bfz'|,\quad\text{and}\quad F(\bfx\bfz) = F(\bfx'\bfz') \quad\Rightarrow\quad F(\bfx\bfy\bfz) = F(\bfx'\bfy\bfz').
$$
\end{enumerate}
Moreover, we may assume that $|\bfy|=1$ in assertion (iii).
\end{proposition}

\begin{proof}
(i) $\Leftrightarrow$ (ii) $\Rightarrow$ (iii). Trivial or straightforward.

(iii) $\Rightarrow$ (i). Follows from repeated applications of the stated condition. To illustrate, suppose that we have $F(x_1x_3)=F(x'_1x'_3)$ for some $x_1x_3x'_1x'_3\in X^4$. Then for any $x_2,x_4\in X$, we have $F(x_1x_2x_3)=F(x'_1x_2x'_3)$, and then $F(x_1x_2x_3x_4)=F(x'_1x_2x'_3x_4)$.
\end{proof}

Recall that a variadic operation $F\colon X^*\to X\cup\{\varepsilon\}$ is said to be \emph{arity-wise range-idempotent} \cite{MarTeh3} if $F(F(\bfx)^{|\bfx|})=F(\bfx)$ for every $\bfx\in X^*$. Clearly, any B-associative or strongly B-associative variadic operation is arity-wise range-idempotent. Actually, it can be shown \cite{MarTeh3} that if an operation $F\colon X^*\to X\cup\{\varepsilon\}$ is B-associative then it is both B-preassociative and arity-wise range-idempotent. The converse result holds whenever $\ran(F_n)\subseteq X$ for every $n\geqslant 1$ (note that this latter condition was wrongly omitted in \cite{MarTeh3}). The following proposition shows that this result still holds if we replace B-associativity and B-preassociativity by their strong versions.

\begin{proposition}\label{prop:sdf576}
If a variadic operation $F\colon X^*\to X\cup\{\varepsilon\}$ is strongly B-associative, then it is both strongly B-preassociative and arity-wise range-idempotent. The converse result holds whenever $\ran(F_n)\subseteq X$ for every $n\geqslant 1$.
\end{proposition}

\begin{proof}
The first result holds trivially by Proposition~\ref{prop:3wordH}. We now prove the converse result by using Proposition~\ref{prop:3wordH} again. Let $\bfx\bfy\bfz\in X^*$. If $\bfx\bfz=\varepsilon$, then \eqref{eq:sadf576} holds trivially. If $\bfx\bfz\neq\varepsilon$, then we have $F(\bfx\bfz)=F(F(\bfx\bfz)^{|\bfx|}F(\bfx\bfz)^{|\bfz|})$ by arity-wise range-idempotence and then \eqref{eq:sadf576} holds by Proposition~\ref{prop:asf8ssa}(iii).
\end{proof}

Various alternative formulations of B-associativity have been given in \cite{MarTeh3}. For instance, we can prove that an operation $F\colon X^*\to X\cup\{\varepsilon\}$ is B-associative if and only if we have $F(\bfx\bfy)=F(F(\bfx)^{|\bfx|}F(\bfy)^{|\bfy|})$ for every $\bfx\bfy\in X^*$. The following proposition provides similar formulations for strong B-associativity.

\begin{proposition}
Let $F\colon X^*\to X\cup\{\varepsilon\}$ be a variadic operation. The following assertions are equivalent.
\begin{enumerate}
\item[(i)] $F$ is strongly B-associative.

\item[(ii)] For every $\bfx\in X^*$ and every $K\subseteq\{1,\ldots,|\bfx|\}$, we have $F(\bfx)=F(\bfx')$, where $\bfx'\in X^{|\bfx|}$ is defined by $\bfx'|_K=F(\bfx|_K)^{|\bfx|_K|}$ and $\bfx'|_{K^c}=F(\bfx|_{K^c})^{|\bfx|_{K^c}|}$.

\item[(iii)] For every $\bfx\bfy\bfz\in X^*$, we have $F(\bfx\bfy\bfz)=F(F(\bfx\bfz)^{|\bfx|}F(\bfy)^{|\bfy|}F(\bfx\bfz)^{|\bfz|})$.
\end{enumerate}
Moreover, we may assume that $|\bfy|\leqslant 1$ in assertion (iii).
\end{proposition}

\begin{proof}
(i) $\Rightarrow$ (ii) $\Rightarrow$ (iii). Trivial.

(iii) $\Rightarrow$ (i). First observe that $F$ is arity-wise range-idempotent (take $\bfy\bfz=\varepsilon$). Let $\bfx\bfy\bfz\in X^*$, with $|\bfy|\leqslant 1$. If $F(\bfx\bfz)=\varepsilon$, then we have
\begin{eqnarray*}
F(F(\bfx\bfz)^{|\bfx|}\bfy F(\bfx\bfz)^{|\bfz|}) &=& F(\bfy) ~=~ F(F(\bfy)^{|\bfy|}) ~=~ F(F(\bfx\bfz)^{|\bfx|}F(\bfy)^{|\bfy|}F(\bfx\bfz)^{|\bfz|})\\
&=& F(\bfx\bfy\bfz).
\end{eqnarray*}
Otherwise, if $F(\bfx\bfz)\in X$, then setting $\bfx'=F(\bfx\bfz)^{|\bfx|}$ and $\bfz'=F(\bfx\bfz)^{|\bfz|}$, we have
\begin{eqnarray*}
F(F(\bfx\bfz)^{|\bfx|}\bfy F(\bfx\bfz)^{|\bfz|}) &=& F(\bfx'\bfy\bfz') ~=~ F(F(\bfx'\bfz')^{|\bfx'|}F(\bfy)^{|\bfy|}F(\bfx'\bfz')^{|\bfz'|}) \\
& = & F(F(\bfx\bfz)^{|\bfx|}F(\bfy)^{|\bfy|}F(\bfx\bfz)^{|\bfz|}) ~=~ F(\bfx\bfy\bfz).
\end{eqnarray*}
In both cases, we have shown that $F$ is strongly B-associative by Proposition~\ref{prop:3wordH}.
\end{proof}

Proposition~\ref{prop:3wordH} states that an operation $F\colon X^*\to X\cup\{\varepsilon\}$ is strongly B-associative if and only if it satisfies the following two conditions:
\begin{enumerate}
\item[(a)] $F(\bfy)=F(F(\bfy)^{|\bfy|})$ for every $\bfy\in X^*$ (arity-wise range-idempotence),

\item[(b)] $F(\bfx y\bfz)=F(F(\bfx\bfz)^{|\bfx|}yF(\bfx\bfz)^{|\bfz|})$ for every $\bfx y\bfz\in X^*$.
\end{enumerate}
Interestingly, this equivalence also shows how a strongly B-associative $\varepsilon$-standard operation $F\colon X^*\to X\cup\{\varepsilon\}$ can be constructed by choosing first $F_1$, then $F_2$, and so forth. In fact, $F_1$ can be chosen arbitrarily provided that it satisfies $F_1\circ F_1=F_1$. Then, if $F_k$ is already chosen for some $k\geqslant 1$, then $F_{k+1}$ can be chosen arbitrarily from among the solutions of the following equations
\begin{eqnarray*}
&& \delta_{F_{k+1}}\circ F_{k+1} ~=~ F_{k+1}{\,},\label{eq:1rtz78z}\\
&& F_{k+1}(\bfx y\bfz) ~=~ F_{k+1}(F_k(\bfx\bfz)^{|\bfx|}yF_k(\bfx\bfz)^{|\bfz|}),\qquad \bfx y\bfz\in X^{k+1}.\label{eq:2rtz78z}
\end{eqnarray*}

In general, finding all the possible functions $F_{k+1}$ is not an easy task. However, we have the following two propositions, which hold for any strongly B-preassociative function and hence for any strongly B-associative operation. The proof of Proposition~\ref{prop:fd87f} is straightforward and thus omitted. Proposition~\ref{prop:s8da6f} was established in \cite{MarTeh3}.

\begin{proposition}\label{prop:fd87f}
Let $F\colon X^*\to Y$ be a B-preassociative (resp.\ strongly B-preassocia{\-}tive) function.
\begin{enumerate}
\item[(a)] If $F_k$ is symmetric for some $k\geqslant 2$, then so is $F_{k+1}$.

\item[(b)] If $F_k$ is constant for some $k\geqslant 1$, then so is $F_{k+1}$.

\item[(c)] For any sequence $(c_k)_{k\geqslant 1}$ in $Y$ and every $n\geqslant 1$, the function $G\colon X^*\to Y$ defined by $G_k=F_k$, if $k\leqslant n$, and $G_k=c_k$, if $k>n$, is B-preassociative (resp.\ strongly B-preassociative).
\end{enumerate}
\end{proposition}

\begin{proposition}[{\cite{MarTeh3}}]\label{prop:s8da6f}
Let $F\colon X^*\to Y$ be a B-preassociative function and let $k\geqslant 2$ be an integer. If the function $\bfy\in X^k\mapsto F_{k+2}(x\bfy z)$ is symmetric for every $xz\in X^2$, then so is the function $\bfy\in X^{k+1}\mapsto F_{k+3}(x\bfy z)$ for every $xz\in X^2$.
\end{proposition}

Proposition~\ref{prop:s8da6f} motivates the question of finding necessary and sufficient conditions on a (strongly) B-preassociative function $F\colon X^*\to Y$ for the following condition to hold:
\begin{equation}\label{eq:75wtbd}
F(abcd) ~=~ F(acbd){\,},\qquad a,b,c,d\in X.
\end{equation}
Not all strongly B-preassociative functions satisfy condition \eqref{eq:75wtbd}. To give a very simple example, just consider the identity function $F=\id_{X^*}$.

However, one can show that condition \eqref{eq:75wtbd} holds for all strongly B-associative operations.

\begin{lemma}\label{lemma:8sf6}
Any strongly B-associative operation $F\colon X^*\to X\cup\{\varepsilon\}$ satisfies condition \eqref{eq:75wtbd}.
\end{lemma}

\begin{proof}
Let $a,b,c,d\in X$ and set $x=F(ab)$ and $y=F(cd)$. Repeated applications of strong B-associativity give
$$
F(abcd) ~=~ F(xxyy) ~=~ F(F(xy)^4) ~=~ F(xyxy) ~=~ F(acbd),
$$
which shows that condition \eqref{eq:75wtbd} holds.
\end{proof}

From Proposition~\ref{prop:s8da6f} and Lemma~\ref{lemma:8sf6} we immediately derive the following corollary, which states that, for every strongly B-associative operation $F\colon X^*\to X\cup\{\varepsilon\}$ and every integer $n\geqslant 4$, the $n$-ary part $F_n$ is invariant under any permutation of its arguments except the first and the last ones.

\begin{corollary}\label{cor:s8da5}
If $F\colon X^*\to X\cup\{\varepsilon\}$ is strongly B-associative, then, for every integer $k\geqslant 1$ and every $xz\in X^2$, the function $\bfy\in X^k\mapsto F_{k+2}(x\bfy z)$ is symmetric.
\end{corollary}

In \cite{MarTeh3}, the authors show how new B-preassociative functions can be constructed from given B-preassociative functions by compositions with unary maps. These results still hold for strongly B-preassociative functions. The proofs are straightforward.

\begin{proposition}[Right composition]
If $F\colon X^*\to Y$ is strongly B-preassociative then, for every function $g\colon X'\to X$, any function $H\colon X'^*\to Y$ such that $H_n=F_n\circ(g,\ldots,g)$ for every $n\geqslant 1$ is strongly B-preassociative.
\end{proposition}

\begin{proposition}[Left composition]\label{prop:leftcomp56}
Let $F\colon X^*\to Y$ be a strongly B-preassociative function and let $(g_n)_{n\geqslant 1}$ be a sequence of functions from $Y$ to $Y'$. If $g_n|_{\ran(F_n)}$ is one-to-one for every $n\geqslant 1$, then any function $H\colon X^*\to Y'$ such that $H_n=g_n\circ F_n$ for every $n\geqslant 1$ is strongly B-preassociative.
\end{proposition}

We now give a factorization result for strongly B-preassociative functions. We first restrict ourselves to strongly B-preassociative functions $F\colon X^*\to Y$ which are \emph{arity-wise quasi-range-idempotent} \cite{MarTeh3}, i.e., such that $\ran(\delta_{F_n})=\ran(F_n)$ for every $n\geqslant 1$. The following theorem gives a characterization of the B-preassociative and arity-wise quasi-range-idempotent functions $F\colon X^*\to Y$ as compositions of the form $F_n=f_n\circ H_n$, where $H\colon X^*\to X\cup\{\varepsilon\}$ is a B-associative $\varepsilon$-standard operation and $f_n\colon \ran(H_n)\to Y$ is one-to-one.

Recall that a function $g$ is a \emph{quasi-inverse} \cite[Sect.~2.1]{SchSkl83} of a function $f$ if
$$
f\circ g|_{\ran(f)}=\id|_{\ran(f)}\quad\mbox{and}\quad\ran(g|_{\ran(f)})=\ran(g).
$$
Recall also that the statement ``every function has a quasi-inverse'' is equivalent to the Axiom of Choice (AC). Throughout this paper we denote the set of all quasi-inverses of $f$ by $Q(f)$.

\begin{theorem}[{\cite{MarTeh3}}]\label{thm:FactoriAWRI-BPA237111}
Assume AC and let $F\colon X^*\to Y$ be a function. The following assertions are equivalent.
\begin{enumerate}
\item[(i)] $F$ is B-preassociative and arity-wise quasi-range-idempotent.

\item[(ii)] There exists a B-associative $\varepsilon$-standard operation $H\colon X^*\to X\cup\{\varepsilon\}$ and a sequence $(f_n)_{n\geqslant 1}$ of one-to-one functions $f_n\colon\ran(H_n)\to Y$ such that $F_n=f_n\circ H_n$ for every $n\geqslant 1$.
\end{enumerate}
If condition (ii) holds, then for every $n\geqslant 1$ we have $F_n=\delta_{F_n}\circ H_n$, $f_n=\delta_{F_n}|_{\ran(H_n)}$, $f_n^{-1}\in Q(\delta_{F_n})$, and we may choose $H_n=g_n\circ F_n$ for any $g_n\in Q(\delta_{F_n})$.
\end{theorem}

Using Propositions~\ref{prop:asf8ssa} and \ref{prop:sdf576}, we can show that Theorem~\ref{thm:FactoriAWRI-BPA237111} can be easily adapted to the strong version of B-preassociativity.

\begin{corollary}\label{cor:fsa7f}
Theorem~\ref{thm:FactoriAWRI-BPA237111} still holds if we replace B-preassociativity with strong B-preassociativity in assertion (i) and B-associativity with strong B-associativity in assertion (ii).
\end{corollary}

\begin{proof}
(i) $\Rightarrow$ (ii). Since $F$ satisfies condition (i) of Theorem~\ref{thm:FactoriAWRI-BPA237111}, it also satisfies condition (ii). Since $H$ is B-associative, it is arity-wise range-idempotent. To see that $H$ is strongly B-associative, by Proposition~\ref{prop:sdf576} it suffices to show that it is strongly B-preassociative. Let $\bfx\bfx'\bfy\bfz\bfz'\in X^*$ such that $|\bfx|=|\bfx'|$, $|\bfz|=|\bfz'|$, $|\bfx\bfz|\geqslant 1$, and $H(\bfx\bfz)=H(\bfx'\bfz')$. Then, we have $f_{|\bfx\bfz|}\circ H(\bfx\bfz)=f_{|\bfx\bfz|}\circ H(\bfx'\bfz')$, that is, $F(\bfx\bfz)=F(\bfx'\bfz')$. By strong B-preassociativity of $F$, we then have $F(\bfx\bfy\bfz)=F(\bfx'\bfy\bfz')$ and hence
$$
H(\bfx\bfy\bfz) ~=~ g_{|\bfx\bfy\bfz|}\circ F(\bfx\bfy\bfz) ~=~ g_{|\bfx\bfy\bfz|}\circ F(\bfx'\bfy\bfz') ~=~ H(\bfx'\bfy\bfz'),
$$
which shows that $H$ is strongly B-preassociative by Proposition~\ref{prop:asf8ssa} (the case when $\bfx\bfz=\varepsilon$ is trivial).

(ii) $\Rightarrow$ (i). $F$ is arity-wise quasi-range-idempotent by Theorem~\ref{thm:FactoriAWRI-BPA237111}. It is also strongly B-preassociative by Proposition~\ref{prop:leftcomp56}.
\end{proof}

We now provide a factorization result for the whole class of strongly B-preassociative functions. It is based on the following characterization of B-preassociative functions.

Recall that a string function $F\colon X^*\to X^*$ is said to be \emph{associative} \cite{LehMarTeh} if it satisfies the equation $F(\bfx\bfy\bfz)=F(\bfx F(\bfy)\bfz)$ for every $\bfx\bfy\bfz\in X^*$. It is said to be \emph{length-preserving} \cite{MarTeh4} if $|F(\bfx)|=|\bfx|$ for every $\bfx\in X^*$.

\begin{theorem}[{\cite{MarTeh4}}]\label{thm:fa7sfds}
Assume AC and let $F\colon X^*\to Y$ be a function. The following assertions are equivalent.
\begin{enumerate}
\item[(i)] $F$ is B-preassociative.

\item[(ii)] There exist an associative and length-preserving function $H\colon X^*\to X^*$ and a sequence $(f_n)_{n\geqslant 1}$ of one-to-one functions $f_n\colon\ran(H_n)\to Y$ such that $F_n=f_n\circ H_n$ for every $n\geqslant 1$.
\end{enumerate}
If condition (ii) holds, then for every $n\geqslant 1$ we have $f_n=F|_{\ran(H_n)}=F_n|_{\ran(H_n)}$, $f_n^{-1}\in Q(F_n)$, and we may choose $H_n=g_n\circ F_n$ for any $g_n\in Q(F_n)$.
\end{theorem}

Proceeding as in the proof of Corollary~\ref{cor:fsa7f}, from Theorem~\ref{thm:fa7sfds} we easily derive the following characterization of the class of strongly B-preassociative functions.

\begin{corollary}\label{cor:fdfas}
Theorem~\ref{thm:fa7sfds} still holds if we replace B-preassociativity with strong B-preassociativity in assertion (i) and add the condition that $H$ is strongly B-preassociative in assertion (ii).
\end{corollary}

Clearly, Corollary~\ref{cor:fdfas} motivates the problem of characterizing the class of those string functions which are associative, length-preserving, and strongly B-preassociative.

We end this section by an investigation of those strong B-associative functions which are invariant by replication. Recall that a variadic function $F\colon X^*\to Y$ is \emph{invariant by replication} \cite{MarTeh2} if for every $\bfx\in X^*$ and every $k\geqslant 1$ we have $F(\bfx^k)=F(\bfx)$.

\begin{definition}
We say that a variadic function $F\colon X^*\to Y$ has a \emph{multiplicatively growing range} if $\ran(F_n)\subseteq\ran(F_{kn})$ for any $k,n\geqslant 1$.
\end{definition}

The following proposition is a simultaneous generalization of several results reported in \cite{Mar00} and \cite[pp. 38-41]{GraMarMesPap09}.

\begin{proposition}\label{prop:SBA-MNE-TFAE4}
Let $F\colon X^*\to X\cup\{\varepsilon\}$ be a strongly B-associative operation. The following assertions are equivalent.
\begin{enumerate}
\item[(i)] $F$ has a multiplicatively growing range.

\item[(ii)] $F$ is invariant by replication.

\item[(iii)] For any $k,n\geqslant 1$ and any $\bfx^{(1)}\cdots\bfx^{(n)}\in X^*$, we have
$$
F((\bfx^{(1)})^k\cdots(\bfx^{(n)})^k) ~=~ F(\bfx^{(1)}\cdots\bfx^{(n)}).
$$
\end{enumerate}
Moreover, if any of these conditions hold, then
\begin{enumerate}
\item[(a)] For any $n\geqslant 1$ and any $\bfx^{(1)}\cdots\bfx^{(n)}\in X^*$ such that $|\bfx^{(1)}|=\cdots = |\bfx^{(n)}|\geqslant 1$, we have
$$
F(F(\bfx^{(1)})\cdots F(\bfx^{(n)})) ~=~  F(\bfx^{(1)}\cdots\bfx^{(n)}).
$$

\item[(b)] For any $n\geqslant 1$ and any $\bfx=x_1\cdots x_n\in X^n$, we have
$$
F(x_1\cdots{\,}x_n) ~=~ F(x'_n\cdots{\,}x'_1),
$$
where $x'_k=F(\bfx_{-k})$ and $\bfx_{-k}$ is obtained from $\bfx$ by removing the letter $x_k$.

\item[(c)] $F$ is strongly bisymmetric, i.e., for every $p$-by-$n$ matrix whose entries are in $X$, we have
    $$
    F(F(\mathbf{r}_1)\cdots F(\mathbf{r}_p)) ~=~  F(F(\mathbf{c}_1)\cdots F(\mathbf{c}_n)),
    $$
    where $\mathbf{r}_1,\ldots,\mathbf{r}_p$ denote the rows and $\mathbf{c}_1,\ldots,\mathbf{c}_n$ denote the columns of the matrix.
\end{enumerate}
\end{proposition}

\begin{proof}
(iii) $\Rightarrow$ (ii). Taking $n=1$, we see that $F$ is invariant by replication.

(ii) $\Rightarrow$ (i). Let $k,n\geqslant 1$ and $\bfx\in X^n$. Then $F(\bfx)=F(\bfx^k)\in\ran(F_{kn})$. Therefore, $F$ has a multiplicatively growing range.

(i) $\Rightarrow$ (iii). Let us first show that
\begin{equation}\label{eq:RMAI}
F(F(\bfx)^{k|\bfx|})=F(\bfx),\qquad\bfx\in X^*,~k\geqslant 1.
\end{equation}
Let $m\geqslant 0$, $k\geqslant 1$, and take $\bfx\in X^m$ and $\bfz\in X^{km}$ such that $F(\bfx)=F(\bfz)$. Since $F$ is arity-wise range-idempotent by Proposition~\ref{prop:sdf576}, it follows that $F(F(\bfx)^{km})=F(F(\bfz)^{km})=F(\bfz)=F(\bfx)$, which proves \eqref{eq:RMAI}.

Then, for any $n\geqslant 1$ and any $\bfx^{(1)}\cdots\bfx^{(n)}\in X^*$ we have
\begin{eqnarray*}
F((\bfx^{(1)})^k\cdots(\bfx^{(n)})^k) &=& F(F(\bfx^{(1)}\cdots\bfx^{(n)})^{k|\bfx^{(1)}\cdots\bfx^{(n)}|})\qquad\mbox{(strong B-associativity)}\\
&=& F(\bfx^{(1)}\cdots\bfx^{(n)})\qquad\mbox{(by condition \eqref{eq:RMAI})}.
\end{eqnarray*}

Let us now show that conditions (a), (b), and (c) hold.

(a) Setting $k=|\bfx^{(1)}|=\cdots = |\bfx^{(n)}|\geqslant 1$ we have
\begin{eqnarray*}
F(\bfx^{(1)}\cdots\bfx^{(n)}) &=& F(F(\bfx^{(1)})^k\cdots F(\bfx^{(n)})^k)\qquad\mbox{(B-associativity)}\\
&=& F(F(\bfx^{(1)})\cdots F(\bfx^{(n)}))\qquad\mbox{(by (iii))}.
\end{eqnarray*}

(b) We have
\begin{eqnarray*}
F(x_1\cdots x_n) &=& F(x_1^{n-1}\cdots x_n^{n-1})\qquad\mbox{(by (iii))}\\
&=& F((x'_n\cdots x'_1)^{n-1})\qquad\mbox{(strong B-associativity)}\\
&=& F(x'_n\cdots x'_1)\qquad\mbox{(by (ii))}.
\end{eqnarray*}

(c) We have
\begin{eqnarray*}
F(F(\mathbf{r}_1)\cdots F(\mathbf{r}_p)) &=& F(\mathbf{r}_1\cdots \mathbf{r}_p)\qquad\mbox{(by (a))}\\
&=& F((F(\mathbf{c}_1)\cdots F(\mathbf{c}_n))^p)\qquad\mbox{(strong B-associativity)}\\
&=& F(F(\mathbf{c}_1)\cdots F(\mathbf{c}_n))\qquad\mbox{(by (ii))}.
\end{eqnarray*}
The proof is now complete.
\end{proof}

\begin{remark}
Let $F\colon X^*\to X\cup\{\varepsilon\}$ be a strongly B-associative operation having a multiplicative growing range and such that $\ran(F_n)\subseteq X$ for every $n\geqslant 1$. Proposition~\ref{prop:SBA-MNE-TFAE4}(a) enables us to translate certain functional conditions involving $F$ and letters in $X$ into similar functional conditions involving $F$ and strings of the same length. To illustrate, starting from the condition
\begin{equation}\label{eq:8x6xcv}
F(xyz) ~=~ F(F(xz)y F(xz)),\qquad xyz\in X^3,
\end{equation}
which holds by Eq.~\eqref{eq:sadf576}, we derive the condition
$$
F(\bfx\bfy\bfz) ~=~ F(F(\bfx\bfz)F(\bfy)F(\bfx\bfz)),\qquad \bfx\bfy\bfz\in X^*,\quad |\bfx|=|\bfy|=|\bfz|\geqslant 1.
$$
Indeed, it suffices to set $x=F(\bfx)$, $y=F(\bfy)$, and $z=F(\bfz)$ in \eqref{eq:8x6xcv} and apply Proposition~\ref{prop:SBA-MNE-TFAE4}(a) to the resulting condition.
\end{remark}

\section{Strongly B-preassociative mean functions}

In this final section we recall a variant of Kolmogoroff-Nagumo's characterization of the class of quasi-arithmetic means based on the strong B-associativity property. We also generalize this characterization to strongly B-preassociative functions.

Let $\mathbb{I}$ be a nontrivial real interval (i.e., nonempty and not a singleton), possibly unbounded. Recall that a function $F\colon \mathbb{I}^*\to\R$ is said to be a \emph{quasi-arithmetic pre-mean function} \cite{MarTeh3} if there exist continuous and strictly increasing functions $f\colon\mathbb{I}\to\R$ and $f_n\colon\R\to\R$ $(n\geqslant 1)$ such that
$$
F_n(\bfx) ~=~ f_n\bigg(\frac{1}{n}\sum_{i=1}^n f(x_i)\bigg),\qquad n\geqslant 1.
$$
This function is said to be a \emph{quasi-arithmetic mean function} (see, e.g., \cite[Sect.~4.2]{GraMarMesPap09}) if $f_n=f^{-1}$ for every $n\geqslant 1$. In this case we have $\ran(F_n)\subseteq\mathbb{I}$ for every $n\geqslant 1$.

Thus defined, the class of quasi-arithmetic pre-mean functions includes all the quasi-arithmetic mean functions. Actually the quasi-arithmetic mean functions are exactly those quasi-arithmetic pre-mean functions which are \emph{idempotent}, that is, such that $\delta_{F_n}=\id_{\mathbb{I}}$ for every $n\geqslant 1$. However, there are also many non-idempotent quasi-arithmetic pre-mean functions. Taking for instance $f_n(x)=nx$ and $f(x)=x$ over the reals $\mathbb{I}=\R$, we obtain the sum function. Taking $f_n(x)=\exp(nx)$ and $f(x)=\ln(x)$ over $\mathbb{I}=\left]0,\infty\right[$, we obtain the product function.

The following proposition \cite{MarTeh3} shows that the generators $f_n$ and $f$ of any quasi-arithmetic pre-mean function are defined up to an affine transformation.

\begin{proposition}[{\cite{MarTeh3}}]\label{prop:genfg}
Let $\mathbb{I}$ be a nontrivial real interval, possibly unbounded. Let $f,g\colon\mathbb{I}\to\R$ and $f_n,g_n\colon\R\to\R$ $(n\geqslant 1)$ be continuous and strictly monotonic functions. Then the functions $f_n(\frac{1}{n}\sum_{i=1}^n f(x_i))$ and $g_n(\frac{1}{n}\sum_{i=1}^n g(x_i))$ coincide on $\mathbb{I}^n$ if and only if there exist $r,s\in\R$, $r\neq 0$, such that $g_n^{-1}\circ f_n=g\circ f^{-1}=r{\,}\id+s$ for every $n\geqslant 1$.
\end{proposition}

We now recall the characterization of the class of quasi-arithmetic mean functions as given by Kolmogoroff \cite{Kol30} and Nagumo \cite{Nag30}. The following theorem gives the characterization following Kolmogoroff (we set $F(\varepsilon)$ to an arbitrary value in $\mathbb{I}$, see Remark~\ref{rem:a8d5}). Nagumo's characterization is the same except that the strict increasing monotonicity of each function $F_n$ is replaced with the strict internality of $F_2$ (i.e., $x<y$ implies $x<F_2(x,y)<y$).

\begin{theorem}[Kolmogoroff-Nagumo]\label{thm:KolNag30}
Let $\mathbb{I}$ be a nontrivial real interval, possibly unbounded. A variadic function $F\colon\mathbb{I}^*\to\mathbb{I}$ is B-associative and, for every $n\geqslant 1$, the $n$-ary part $F_n$ is symmetric, continuous, idempotent, and strictly increasing in each argument if and only if $F$ is a quasi-arithmetic mean function.
\end{theorem}

As recently observed by the authors \cite{MarTeh3}, idempotence can be removed from the assumptions of Theorem~\ref{thm:KolNag30}. Indeed, if a B-associative function $F\colon\mathbb{I}^*\to\mathbb{I}$ is such that $\delta_{F_n}$ is one-to-one for some $n\geqslant 1$, then necessarily $\delta_{F_n}=\id_{\mathbb{I}}$. This observation immediately follows from the identity $\delta_{F_n}=\delta_{F_n}\circ\delta_{F_n}$, which holds whenever $F$ is B-associative.

In the following theorem, we show that Kolmogoroff-Nagumo's characterization still holds if we replace both B-associativity and symmetry with strong B-associativity. This result was already established in \cite{Mar00}. However, here we provide an alternative proof based on Kolmogoroff's ideas. Here again, idempotence is redundant.

We first consider a lemma which generalizes the result reported in \cite[Lemma~4.9]{GraMarMesPap09}.

\begin{lemma}\label{lemma:psi}
Let $F\colon X^*\to X\cup\{\varepsilon\}$ be a strongly B-associative operation having a multiplicatively growing range. Then, for any $\mathbf{a},\mathbf{b}\in X^*\setminus\{\varepsilon\}$, there exists a function $\psi\colon [0,1]\cap\Q\to X\cup\{\varepsilon\}$, namely
$$
\psi(p/q) ~=~ F(\mathbf{b}^p\mathbf{a}^{q-p}),\qquad p/q\in [0,1]\cap\Q,
$$
with $\psi(0)=F(\mathbf{a})$ and $\psi(1)=F(\mathbf{b})$, such that for every $n\geqslant 1$ and every $\bfz\in [0,1]^n\cap\Q^n$ such that $z_1\neq 0$ and $z_n\neq 1$, we have
$$
F(\psi(z_1)\cdots\psi(z_n)) ~=~ \psi\bigg(\frac{1}{n}\,\sum_{i=1}^nz_i\bigg).
$$
\end{lemma}

\begin{proof}
We first observe that $\psi$ is a well-defined function. Indeed, if $p/q=p'/q'$ are two representations of the same rational, then we have
\begin{eqnarray*}
F(\mathbf{b}^p\mathbf{a}^{q-p}) &=& F(\mathbf{b}^{p'p}\mathbf{a}^{p'(q-p)})\qquad\mbox{(by Proposition~\ref{prop:SBA-MNE-TFAE4}(iii))}\\
&=& F(\mathbf{b}^{pp'}\mathbf{a}^{p(q'-p')})\\
&=& F(\mathbf{b}^{p'}\mathbf{a}^{q'-p'})\qquad\mbox{(by Proposition~\ref{prop:SBA-MNE-TFAE4}(iii))}.
\end{eqnarray*}

Now, for any $z_1=p_1/q,\ldots,z_n=p_n/q$, with $p_i\leqslant q$, $p_1\neq 0$ and $p_n\neq q$, we have
\begin{eqnarray*}
F(\psi(z_1)\cdots\psi(z_n)) &=& F(F(\mathbf{b}^{p_1}\mathbf{a}^{q-p_1})\cdots F(\mathbf{b}^{p_n}\mathbf{a}^{q-p_n}))\\
&=& F(\mathbf{b}^{p_1}\mathbf{a}^{q-p_1}\cdots \mathbf{b}^{p_n}\mathbf{a}^{q-p_n})\qquad\mbox{(by Proposition~\ref{prop:SBA-MNE-TFAE4}(a))}\\
&=& F(\mathbf{b}^{\sum p_i}\mathbf{a}^{nq-\sum p_i})\qquad\mbox{(by Corollary~\ref{cor:s8da5})}\\
&=& \psi\bigg(\frac{1}{nq}\sum_{i=1}^n p_i\bigg) ~=~ \psi\bigg(\frac{1}{n}\,\sum_{i=1}^nz_i\bigg).
\end{eqnarray*}
This completes the proof of the lemma.
\end{proof}

\begin{theorem}\label{thm:9asf7}
Theorem~\ref{thm:KolNag30} still holds if we replace B-associativity and symmetry with strong B-associativity. Also, idempotence can be removed.
\end{theorem}

\begin{proof}
(Necessity) Let $F\colon \mathbb{I}^*\to\mathbb{I}$ be a strongly B-associative function such that, for every $n\geqslant 1$, the function $F_n$ is continuous, idempotent, and strictly increasing in each argument.

We first assume that $\mathbb{I}$ is a closed interval $[a,b]$, with $b>a$. Since $F_n$ is idempotent for every $n\geqslant 1$, $F$ has a multiplicative growing range. By Lemma~\ref{lemma:psi} the function $\psi\colon [0,1]\cap\Q\to [a,b]$ defined by $\psi(p/q) = F(b^pa^{q-p})$ is well defined and such that for every $n\geqslant 1$ and every $\bfz\in [0,1]^n\cap\Q^n$ such that $z_1\neq 0$ and $z_n\neq 1$, we have
\begin{equation}\label{eq:s5fsf}
F(\psi(z_1)\cdots\psi(z_n)) ~=~ \psi\bigg(\frac{1}{n}\sum_{i=1}^nz_i\bigg).
\end{equation}
Moreover, it is easy to see that the function $\psi$ is strictly increasing.

Let us now show that the restriction of $\psi$ to $\left]0,1\right[\cap\Q$ can be extended to a continuous function $\overline{\psi}\colon\left]0,1\right[\to [a,b]$.

Let $x\in\left]0,1\right[$. Since $\psi$ is nondecreasing we can define
$$
u ~=~ \lim_{z\to x^-}\psi(z)\quad\mbox{and}\quad v ~=~ \lim_{z\to x^+}\psi(z).
$$
Let us show that $u=v$. For contradiction, assume that $u<v$ and consider two sequences $(z_m)_{m\geqslant 1}$ and $(z'_m)_{m\geqslant 1}$ in $[0,1]\cap\Q$ such that $z_m\to x^-$, $z'_m\to x^+$, and $(z_m+z'_m)/2<x$. Using \eqref{eq:s5fsf} and the continuity of $F$, we then have
$$
u ~=~ \lim_{m\to\infty}\psi\Big(\frac{z_m+z'_m}{2}\Big) ~=~ \lim_{m\to\infty} F(\psi(z_m),\psi(z'_m)) ~=~ F(u,v) ~<~ F(u,u) ~=~ u,
$$
a contradiction.

Let $\overline{\psi}\colon\left]0,1\right[\to\left]\alpha,\beta\right[$ be the continuous extension of $\psi$, where $\left]\alpha,\beta\right[=\ran(\overline{\psi})$. Let us show that $\alpha =a$. Due to the uniqueness of the limit, we have $\lim_{z\to 0^+}\psi(z)=\lim_{t\to 0^+}\overline{\psi}(t)=\alpha$. Then, using \eqref{eq:s5fsf} and the continuity of $F$, we have
\begin{eqnarray*}
F(b,\alpha) &=& \lim_{z\to 0^+}F(\psi(1),\psi(z)) ~=~ \lim_{z\to 0^+}\psi\Big(\frac{1+z}{2}\Big) ~=~ \psi\Big(\frac{1}{2}\Big) ~=~ F(\psi(1),\psi(0))\\
&=& F(b,a).
\end{eqnarray*}
Since $F_2$ is one-to-one in its second argument, we must have $\alpha =a$. We prove similarly that $\beta =b$. Thus, $\overline{\psi}$ can be further extended to a continuous and strictly increasing function from $[0,1]$ onto $[a,b]$. Denoting by $f\colon [a,b]\to [0,1]$ the inverse of this continuous extension, from \eqref{eq:s5fsf} and continuity we derive the identity
$$
F_n(\bfx) ~=~ f^{-1}\bigg(\frac{1}{n}\sum_{i=1}^n f(x_i)\bigg),\qquad\bfx\in [a,b]^n,~n\geqslant 1,
$$
which proves the result when $\mathbb{I}=[a,b]$.

Let us now prove the result for a general nontrivial interval $\mathbb{I}$. Here we use arguments sketched in \cite[Theorem~4.10]{GraMarMesPap09}. Let $M_f\colon\mathbb{I}^*\to\mathbb{I}$ denote the quasi-arithmetic mean function generated by $f\colon\mathbb{I}\to\R$. Let $a=\inf\mathbb{I}$ and $b=\sup\mathbb{I}$. Let also $(a_m)_{m\geqslant 1}$ (resp.\ $(b_m)_{m\geqslant 1}$) be a strictly decreasing (resp.\ strictly increasing) sequence in $\mathbb{I}$ converging to $a$ (resp.\ $b$). From the previous result it follows that there exist continuous and strictly increasing functions $f_m\colon [a_m,b_m]\to\R$ and $f_{m+1}\colon [a_{m+1},b_{m+1}]\to\R$ such that $F=M_{f_m}$ on $[a_m,b_m]^*$ and $F=M_{f_{m+1}}$ on $[a_{m+1},b_{m+1}]^*$. By Proposition~\ref{prop:genfg}, we have $M_{f_{m+1}}=M_{rf_{m+1}+s}$ for all $r,s\in\R$, with $r\neq 0$. It follows that $f_{m+1}$ can be chosen so that $f_{m+1}(a_m)=f_m(a_m)$ and $f_{m+1}(b_m)=f_m(b_m)$. Since $M_{f_{m+1}}=F=M_{f_m}$ on $[a_m,b_m]^*$ from Proposition~\ref{prop:genfg} it follows that there exist $c,d\in\R$, with $c\neq 0$ such that $f_m=cf_{m+1}+d$. Due to the definition of $f_{m+1}$, we must have $c=1$ and $d=0$, that is, $f_{m+1}=f_m$ on $[a_m,b_m]$.

Define $f\colon\mathbb{I}\to\R$ by
$$
f(x) ~=~
\begin{cases}
~\lim_{m\to\infty}f_m(x), & \mbox{if $x\in\left]a,b\right[$},\\
~\inf_{m\geqslant 1} f_m(a_m), & \mbox{if $x=a\in\mathbb{I}$},\\
~\sup_{m\geqslant 1} f_m(b_m), & \mbox{if $x=b\in\mathbb{I}$}.
\end{cases}
$$
It is clear that $f$ is continuous and strictly increasing. Moreover, we have $F=M_f$ on $\bigcup_m [a_m,b_m]^*$ and even on $\mathbb{I}^*$ by continuity.

(Sufficiency) Straightforward.
\end{proof}

In \cite{MarTeh3} the authors established a generalization of Kolmogoroff-Nagumo's characterization to quasi-arithmetic pre-mean functions. In the next two theorems we state this result and show that its assumptions can be weakened by replacing both B-preassociativity and symmetry with strong B-preassociativity.

\begin{theorem}[{\cite{MarTeh3}}]\label{thm:KolmExt}
Let $\mathbb{I}$ be a nontrivial real interval, possibly unbounded. A function $F\colon \mathbb{I}^*\to\R$ is B-preassociative and, for every $n\geqslant 1$, the function $F_n$ is symmetric, continuous, and strictly increasing in each argument if and only if $F$ is a quasi-arithmetic pre-mean function.
\end{theorem}

\begin{theorem}
Theorem~\ref{thm:KolmExt} still holds if we replace B-preassociativity and symmetry with strong B-preassociativity.
\end{theorem}

\begin{proof}
We note that the proof is very similar to that of Theorem~\ref{thm:KolmExt} (see \cite{MarTeh3}).

(Necessity) Since $F_n$ is increasing and continuous for every $n\geqslant 1$, it follows that $F$ is arity-wise quasi-range-idempotent. Let $H\colon \mathbb{I}^*\to \mathbb{I}\cup\{\varepsilon\}$ be the $\varepsilon$-standard operation defined by $H_n=\delta_{F_n}^{-1}\circ F_n$ for every $n\geqslant 1$. It is clear that every $H_n$ is continuous, idempotent, and strictly increasing. By Corollary~\ref{cor:fsa7f} (here AC is not needed since $\delta_{F_n}^{-1}$ is an inverse), $H$ is strongly B-associative (and remains so if we modify the value of $H(\varepsilon)$ into any element of $\mathbb{I}$; see Remark~\ref{rem:a8d5}). By Theorem~\ref{thm:9asf7} it follows that $H$ is a quasi-arithmetic mean function. This completes the proof.

(Sufficiency) Straightforward.
\end{proof}

\section{Concluding remarks and open problems}

We have investigated the strong B-associativity property for variadic operations and introduced a relaxation of this property, namely strong B-preassociativity. In particular, we have presented a characterization of the class of strongly B-preassociative functions in terms of associative string functions.

We end this paper with the following questions:
\begin{enumerate}
\item[(a)] Find necessary and sufficient conditions on a (strongly) B-preassociative function $F\colon X^*\to Y$ for condition \eqref{eq:75wtbd} to hold.

\item[(b)] Find necessary and sufficient conditions on a B-associative operation $F\colon X^*\to X\cup\{\varepsilon\}$ satisfying $F(xyz)=F(F(xz)yF(xz))$ for every $xyz\in X^3$ to be strongly B-associative. What if $F$ satisfies the symmetry condition stated in Corollary~\ref{cor:s8da5}?

\item[(c)] Similarly, find necessary and sufficient conditions on a B-preassociative function $F\colon X^*\to Y$ satisfying the condition
    $$
    F(xz)=F(x'z')\quad\Rightarrow\quad F(xyz)=F(x'yz'),\qquad xx'yzz'\in X^5
    $$
    to be strongly B-preassociative.

\item[(d)] Find a characterization of the class of those string functions which are associative, length-preserving, and strongly B-preassociative (cf.\ Corollary~\ref{cor:fdfas}).
\end{enumerate}

\section*{Acknowledgments}

This research is supported by the internal research project F1R-MTH-PUL-15MRO3 of the University of Luxembourg.


\end{document}